\documentclass[10pt,a4 paper]{article}
\usepackage[T1]{fontenc}
\usepackage[english]{babel}
\usepackage[utf8]{inputenc}
\usepackage{indentfirst}
\usepackage[left=30mm, top=20mm, right=25mm, bottom=25mm, nohead]{geometry}
\usepackage{csquotes}
\usepackage[backend=biber,bibencoding=utf8, style=ieee]{biblatex}
\usepackage{amsfonts}
\usepackage{amsmath}
\usepackage{amssymb}
\usepackage{amsthm}
\usepackage{xcolor}
\usepackage{hyperref}
\newtheorem{thm}{Theorem}[section]
\newtheorem{lmm}[thm]{Lemma}
\newtheorem{prp}[thm]{Proposition}
\newtheorem{cor}[thm]{Corrolary}
\theoremstyle{definition}
\newtheorem{dfn}[thm]{Definition}
\newtheorem{ass}[thm]{Assumption}
\theoremstyle{remark}
\newtheorem{rmk}[thm]{Remark}

\newcommand{\dfnem}[1]{\textit{#1}}

\DeclareMathOperator{\Id}{Id}

\DeclareMathOperator{\Tan}{Tan}

\newcommand{\pp}{{\mathcal{P}}}
\newcommand{\cb}{{C_b}}
\newcommand{\cc}{{C^\infty_c}}
\newcommand{\rr}{{\mathbb{R}}}
\newcommand{\rd}{{\rr^d}}
\newcommand{\rds}{{\rr^{d*}}}
\newcommand{\rn}{{\rr^n}}
\newcommand{\intsr}{\int_{s}^{r}}

\newcommand{\subd}{\partial^-_b}
\newcommand{\supd}{\partial^+_b}
\newcommand{\bd}{\partial_b}

\newcommand{\traj}[3]{\mathsf{X}^{{#1},{#2}}_{m_\cdot}(#3)}

\newcommand{\proj}{{\mathfrak{p}}}

\renewcommand{\div}{\operatorname{div}}
\newcommand{\divx}{\operatorname{div}_x}
\newcommand{\dx}{\nabla_x}
\newcommand{\dm}{\nabla_m}
\newcommand{\dy}{\nabla_y}
\newcommand{\fld}[1]{\frac{\delta {#1}}{\delta m}}

\newcommand{\tran}{{\top}}
\newcommand{\ball}{\mathsf{B}}
\newcommand{\sall}{\mathsf{S}}

\newcommand{\hx}{\hat{x}}
\newcommand{\hy}{\hat{y}}
\newcommand{\hm}{\hat{m}}
\newcommand{\hk}{\check{K}}
\newcommand{\hth}{\check{\theta}}
\newcommand{\hatt}{{\hat{T}}}
\newcommand{\hphi}{\hat{\phi}}
\newcommand{\hpi}{\hat{\pi}}
\newcommand{\mbmu}{{\boldsymbol{\mu}}}

\newcommand{\ds}{\displaystyle}

\title{Lyapunov stability of the equilibrium of the non-local continuity equation}
\author{Yurii Averboukh, Aleksei Volkov}
\begin{document}
%======================================================================================
\maketitle
\begin{abstract}
The paper is concerned with the development of Lyapunov methods for the analysis of equilibrium stability in a dynamical system on the space of probability measures driven by a non-local continuity equation.
We derive sufficient conditions of stability of an equilibrium distribution relying on an analysis of a non-smooth Lyapunov function.
For the linear dynamics we reduce the stability analysis to a study of a quadratic form on a tangent space to the space of probability measures.
These results are illustrated by the studies of the stability of the equilibrium measure for gradient flow in the space of probability measures and Gibbs measure for a system of coupled mathematical pendulums.


\end{abstract}
%======================================================================================
\section{Introduction}
%======================================================================================
The paper is concerned with the study of qualitative properties of a dynamical system in space probability measures driven by the non-local continuity equation
\[
\partial_t m_t + \div (f(x,m_t) m_t) = 0.
\]
This equation describes the dynamics of a system of infinitely many identical particles in the case where the right-hand side depends not only on the position of the particles, but also on their current distribution.
In this case, the function $f$ plays the role of a vector field that determines the motion of each particle.

Notice that the continuity equation occurs within modeling of the systems of charged
particles~\cite{Boffi_1984}, behavior of supermassive black holes~\cite{Tucci_2017}, behavior of large groups of animals~\cite{Botsford_1981}, dynamics of biological processes~\cite{Di_Mario_1993}, dynamics of public opinion~\cite{Pluchino_2005}, etc.

Previously, such properties of the non-local continuity equation as stability w.r.t. parameters ~\cite{Ambrosio_2014} and exponential stability~\cite{Frankowska_2022} were studied.
Also in~\cite{Piccoli_et_al} the issues of stability of the measure carrier and integral stability were studied for the case where the equilibrium distribution under study is a Dirac measure.

In this paper, we examine the stability in the sense of Lyapunov.
This concept was first proposed in the famous works of A.\,M. Lyapunov ~\cite{Lyapunov_1892,Lyapunov_1893} for systems of ordinary differential equations, after which it received numerous applications.
Since then, the concept of Lyapunov stability attracts a great attention (see, in particular,~\cite{Malkin_1949,Krasovskii_1959, Seis_2018, Karafyllis_2020,Shevitz_1994}).

The study of stability is relies on the first and second Lyapunov methods~\cite{Lyapunov_1892,Lyapunov_1893}.
In the first method, a linear approximation of the dynamics is used to characterize the stability of the system.
The second method assumes the existence of a differentiable function with certain properties called a Lyapunov function.
The Lyapunov's methods have received significant development for controlled systems, especially for stabilization~\cite{Clarke_1998,Aubin_2011}.
Note that non-smooth Lyapunov functions are often used in control theory.
In this case, sub-/supergradients are used instead of derivatives.

Recall that the phase space for the non-local continuity equation is the phase space of probability measures that is not linear.
This fact is the main feature of the examined problem.
At the same time, the concept of an intrinsic derivative was introduced for this case in~\cite{Cardaliaguet_2019}, as well as a number of analogues of the concepts of non-smooth analysis~\cite{Ambrosio_2005}.
In particular, among them are strong and weak sub-/superdifferentials.
Notice that the square of the distance in the space of probability measures to a given measure is generally a non-differentiable function.
However, some of the elements of the superdifferentiation can be described explicitly~\cite{Ambrosio_2005}.

To construct an analogue of the second Lyapunov method, this article uses a non-smooth Lyapunov function.
The proposed construction of the second Lyapunov method relies on the definition of the barycentric sub-/superdifferential introduced in Section~\ref{sec.diff} of this article.
Notice that the barycentric sub-/superdifferential can be built on the basis of the strong Frechet sub-/superdifferential proposed in the work~\cite{Ambrosio_2005}.
This method allows, in particular, to find sufficient conditions for the stability of particle systems with dynamics determined by a gradient flow in the space of probabilistic measures.

Based on the second Lyapunov method, we study the stability of a stationary solution of the non-local continuity equation in the case of a vector field linear w.r.t. phase variables.
This result is based on the application of the second method with the Lyapunov function equal to half the square of the distance to the equilibrium position.
As an illustration of this method, the stability of the Gibbs measure in a system of coupled mathematical pendulums is considered.

The rest of the paper is organized as follows.

The section ~\ref{sec.term} provides general definitions and notations used in the article.
Here, we introduce equivalent definitions of the non-local continuity equation and formulate some properties of the solution. Those are proved in Appendix~\ref{app.traj}.

Subsection~\ref{sec.diff} is concerned with generalizations of the concepts of non-smooth analysis to functions defined on the space of probabilistic measures.
We introduce here the concepts of the inner derivative, borrowed from~\cite{Cardaliaguet_2019}, and the barycentric sub-/superdifferentials of functionals of the measure.
Some properties of these objects are also formulated. As above, their proofs are postponed to Application~\ref{app.diff}.

Section~\ref{sec.2.lyap} contains a generalization of Lyapunov's second method of stability analysis of a dynamical system.
In this section, the Lyapunov function is assumed not to be smooth, but only barycentric superdifferentiable in the vicinity of the equilibrium position.
In the subsection~\ref{sec.2.ex} an example of dynamics given by a gradient flow is considered.

Section ~\ref{sec.1.lyap} is devoted to obtaining stability conditions for systems with a linear vector field.
Further, in this section an example of dynamics given by a system of coupled mathematical pendulums is considered.

Appendix~\ref{app.traj} contains proofs of the properties of trajectories generated by the non-local continuity equation.

Appendix~\ref{app.diff} provides proofs of the properties of sub- and superdifferentials formulated in subsection~\ref{sec.diff}.

Finally, the Appendix~\ref{app.bound} provides the boundedness property of barycentric sub-/superdifferentials of locally Lipschitz continuous functions.
%======================================================================================
\section{Preliminaries} \label{sec.term}
%======================================================================================
\subsection{General notion}
%======================================================================================

Recall that the metric space $(X,\rho)$ is called \dfnem{Polish} if it is complete and separable.

In this subsection, we will assume that $X$ is a Polish space, $(Y, \|\cdot\|)$ is a Banach separable space, and $p> 1$.

Throughout the rest of the article, the following designations are used.

\begin{itemize}
\item $\rr_+$ defines the ray $[0,+\infty)$ on the $\rr$.
\item $\rd$ is the space of $d$-dimensional column vectors.
\item $\rds$ is the space of $d$-dimensional string vectors.
\item ${}^\tran$ stands for transpose operation.
\item $\ds\ball_R(x)$ denotes a closed ball of radius $R$ centered at point $x$, i.e.
\[
\ball_R(x) \triangle \{ y \in X:\ \rho(x,y) \leq R \}.
\]
\item$\ds\sall_R(x)$ denotes a sphere of radius $R$ centered at point $x$, i.e.
\[
\sall_R(x) \triangleq \{ y \in X:\ \rho(x,y) = R \}.
\]
\item $\Gamma_T(X)$ stands for the space of curves in the space $X$ in the sense of continuous functions on the segment $[0,T]$, i.e.
\[
\Gamma_T(X) \triangleq C([0,T];X).
\]
\item $e_t: \Gamma_T(X) \to X$ denotes the evaluation operator defined by the rule
\[
e_t\big(x(\cdot)\big) \triangleq x(t), \quad \text{at } t \in [0,T],
\]
\item$\cb(X)$ is the space of bounded continuous functionals defined on the space $X$.
\item$\cc(X)$ is the space of infinitely differentiable functionals defined on a finite-dimensional space $X$ having a compact carrier.
\item $\Id: X \to X$ is the identity operator defined for every $x \in X$ by
\[
\Id(x) = x.
\]
\item If $(\Omega,\mathcal{F})$ and $(\Omega',\mathcal{F}')$ are measurable spaces, $\mu$ is measure on $\mathcal{F}$, mapping $h: \Omega\to \Omega'$ is $\mathcal{F}/\mathcal{F}'$-measurable, then $h\sharp\mu$ is a push-forward measure of the measure $\mu$ under the action of the function $h$, determined by the rule
\[
(h\sharp\mu)(\Upsilon) = \mu(h^{-1}(\Upsilon))
\]
for an arbitrary $\Upsilon\in\mathcal{F}'$.
\item$\ds\pp(X)$ is the space of Borel probability measures over $X$.
\item$L_p(X,\mu;Y)$ is the space of functions from $X$ to $Y$ whose $p$-th degree of norm is integrable by $\mu$.
\item $\varsigma_p(\mu)$ denotes the root of the $p$-th degree of the $p$-th moment of the measure $\mu$, i.e.
\[
\varsigma_p(\mu) \triangleq \left( \int_Y \|x\|^p\ \mu(dx) \right)^{1/p} = \|\Id\|_{L_p(Y,\mu;Y)}.
\]
\item $\pp_p(X)$ is the set of elements $\pp(X)$ with a finite $p$-th moment.
\item If $\{X_i\}_{i=1}^n$ is a finite set of arbitrary Polish spaces, and $x =(x_i)_{i=1}^n\in X_1 \times \dots\times X_n$, then the mapping $p^I$ is a projection operator defined by the rule
\[
\proj^I(x) = (x_i)_{i \in I}.
\]
\end{itemize}

Below, we will need the following obvious property of push-forwarded measure.
\begin{prp} \label{prp.bound.mom}
Let $m \in \pp_p(Y)$, $\Phi \in L_p(Y,m;Y)$ and $\mu = (\Id + \Phi)\sharp m$.
Then
\[
\varsigma_p(\mu) \leq \varsigma_p(m) + \|\Phi\|_{L_p(Y,m;Y)}.
\]
\end{prp}

\begin{dfn}[{\cite[Definition 8.4.1]{Ambrosio_2005}}]
Let $\mu \in \pp_p(\rd)$.
Then the \dfnem{tangent space} to the space $\pp_p(\rd)$ at the point $\mu$ is the closure of the set of gradients of test functions in the space $L_p(\rd,\mu;\rds$, i.e.
\[
\Tan(\mu) \triangleq \overline{\big\{ \nabla \varphi(x) \in \cc(X) \big\}}^{L_p(\rd,\mu;\rds)}.
\]
\end{dfn}

\begin{dfn}[{\cite[Problem 1.2 and Section 5.1]{Santambrogio_2015}}]
A \dfnem{plan} between the measures $\mu \in \pp(X)$ and $\nu\in \pp(Y)$ is a measure $\pi \in \pp(X\times Y)$ that
\[
\pi(A \times Y) = \mu(A), \quad\pi(X\times B) = \nu(B),
\]
for arbitrary Borel sets $A\subseteq X$ and $B\subseteq Y$.

The set of all plans is denoted as $\Pi(\mu,\nu)$.
\end{dfn}

\begin{dfn}[{\cite[Problem 1.2 and Section 5.1]{Santambrogio_2015}}] \label{dfn.w}
\dfnem{The Wasserstein metric} between the measures $\mu\in\pp_p(X)$ and $\nu\in\pp_p(X)$ is the function
\[
W_p(\mu,\nu) \triangleq \left(\inf\limits_{\pi\in\Pi(\mu,\nu)} \int\limits_{X \times X} \rho(x,y)^p\ \pi(dxdy) \right)^{1/p}.
\]

The plan $\pi$ providing the infimum is called \dfnem{optimal}; the set of all such plans is denoted by $\Pi_o(\mu,\nu)$.

If the optimal plan can be represented in the form of $(\Id,P)\sharp\mu$, then the mapping $P$ is called an \dfnem{optimal transport}.
\end{dfn}

\begin{rmk} \label{rmk.plan.wp}
Some important properties of the Wasserstein metric are proved at~\cite{Santambrogio_2015}.
\begin{enumerate}
\item Set $\Pi_o$ is always non-empty (\cite[Problem~1.2 and Theorem~1.7] {Santambrogio_2015}).
\item The Wasserstein metric is actually a metric in the space $\pp_p(X)$(\cite[Proposiiton~5.1]{Santambrogio_2015}).
\item Metric space $(\pp_p(X), W_p)$ is Polish (\cite[Theorem~5.11]{Santambrogio_2015}).
\end{enumerate}
\end{rmk}

\begin{rmk} \label{rmk.balls}
Due to the equality
\[
W_p(\mu,\delta_0) = \varsigma_p(\mu),
\]
the boundedness of a set in the space of measures, in the sense of embedding in some ball, is equivalent to the uniform boundedness of the $\varsigma_p(\mu)$ for all $\mu$ from this set.
\end{rmk}

\begin{dfn}
The function $\phi:\pp_p(X)\to\rr$ is called \dfnem{locally Lipschitz continuous} if for any $\alpha>0$ there exists $K_\alpha>0$ such that, for all measures $m_1,m_2\in\pp_p(X)$ with the property $\varsigma_p(m_i)\leq\alpha$, the following inequality holds:
\[
|\phi(m_1) - \phi(m_2)| \leq K_\alpha \cdot W_p(m_1,m_2).
\]
We can assume that $K_\alpha$ is a non-decreasing function of $\alpha$ that is right continuous and has left limit (c\`{a}dl\`{a}g).

Note that the function $K_\alpha$ is bounded on a compact set of values of $\alpha$.
\end{dfn}

\begin{dfn}[{\cite[Definition 10.4.2]{Bogachev_2003}}]
Let $\pi \in \pp(X_1 \times \dots \times X_n)$ and $\proj^k\sharp \pi \in \pp(X_k)$ for some $k$.
Denote
\[
\hat{X}_k \triangleq X_1 \times \dots X_{k-1} \times X_{k+1} \times \dots \times X_n.
\]
Then a family of measures $\{ \hpi^k(\ \cdot\ | x_k) : x_k \in X_k \} \subseteq \pp(\hat{X}_k)$ is called a \dfnem{disintegration of measure $\pi$} by $k$-th variable, if for every test function $\phi\in\cb(X_1\times\dots\times X_n)$ the following equation is fulfilled
\[
\begin{split}
\int\limits_{X_1 \times \dots \times X_n} &\phi(x_1,\dots,x_n)\ \pi(dx_1 \dots dx_n)\\
&= \int\limits_{X_k} \left( \int_{\hat{X}_k} \phi(x_1,\dots,x_n)\ \hpi^k(dx_1 \dots dx_{k-1} dx_{k+1} \dots dx_n | x_k) \right)\ (\proj^k\sharp\pi)(dx_k),
\end{split}
\]
\end{dfn}

\begin{rmk}
If $X_i = \rd$ for all $i$, then, from ~\cite[Corollary~10.4.13]{Bogachev_2003}, it follows that disintegration exist for any measures from $\pp(X_1\times\dots\times X_n)$ w.r.t. any variable.
\end{rmk}
%======================================================================================
\subsection{Non-local continuity equation}
%=====================================================================================
The main object of this paper is the initial value problem for the non-local continuity equation
\begin{equation}\label{eq.sys}
\partial_t m_t + \div (f(x,m_t) m_t) = 0,
\end{equation}
\begin{equation}\label{eq.sys.data}
m_0 = m_*,
\end{equation}
where $f: \rd \times \pp_p(\rd) \to \rd$ is vector field.
Hereinafter $p > 1$ is fixed parameter.
As it was mentioned above, this equation describes a system of particles with dynamics of each particle given by the equation
\begin{equation}\label{eq.sys.x}
\dot{x} = f(x,m_t).
\end{equation}

Everywhere below, we will assume the following conditions.

\begin{ass} \label{ass.lip}
The function $f$ is Lipschitz continuous, i.e. there exist a constant $C_0$ such that, for all $x,y \in \rd$ and $\mu,\nu \in \pp_p(\rd)$,
\[
\|f(x,\mu)-f(y,\nu)\| \leq C_0 \big(\|x-y\| + W_p(\mu,\nu)\big).
\]
\end{ass}

\begin{rmk} \label{ass.ulin}
Substituting $\nu=\delta_0$ and $y=0$ in Assumption~\ref{ass.lip}, we can obtain the the sublinear growth property of the function $f$, i.e. there exists a constant $C_1$ such that, for all $x \in \rd$ and $\mu \in \pp_p(\rd)$,
\[
\|f(x,\mu)\| \leq C_1\big(1 + \|x\| + \varsigma_p(\mu)\big).
\]
\end{rmk}

Next, we introduce two definitions of a solution of the non-local continuity equation~\eqref{eq.sys} on the interval $[0,T]$.
Naturally, the solution of the initial value problem~\eqref{eq.sys },~\eqref{eq.sys.data} is a solution of the continuity equation~\eqref{eq.sys} that satisfies initial condition~\eqref{eq.sys.data}.

\begin{dfn}\label{dfn.sol.dist}
A measure-valued function $m_\cdot$ is a \dfnem{solution of non-local continuity equation~\eqref{eq.sys}} on the interval $[0,T]$ in distributional sense provided that
\begin{equation} \label{eq.sol.dist}
\int_0^T \int_{\rd} \big( \partial_t \varphi(x,t) + \dx \varphi(x,t) \cdot f(x,m_t) \big)\ m_t(dx)dt = 0
\end{equation}
for each test function $\varphi \in \cc(\rd \times (0,T);\rr)$.
\end{dfn}

\begin{dfn}\label{dfn.sol.kant}
A measure-valued function $m_\cdot$ is a \dfnem{solution of non-local continuity equation~\eqref{eq.sys}} on the interval $[0,T]$ in Kantorovich sense, if there exists a measure $\eta \in \pp_p\big(\Gamma_T(\rd)\big)$ such that
\begin{enumerate}
\item $m_t = {e_t}\sharp \eta$,
\item $\eta$-a.a. $x(\cdot) \in \Gamma_T(\rd)$ satisfies
\[
\frac{d}{dt}x(t) = f(x(t),m_t).
\]
\end{enumerate}
\end{dfn}

From \cite[Proposition~8.2.1]{Ambrosio_2005} it follows that Definition~\ref{dfn.sol.dist} and Definition~\ref{dfn.sol.kant} are equivalent.
The existence and uniqueness result for the solution is proved in~\cite{Averboukh_2022} under more general assumptions.

\begin{rmk}
Due to the fact that the solution exists for each $T > 0$, while the solution on a larger interval is an extension of the solution on a smaller one, we can assume that the solution is defined on the entire semiaxis $\rr_+$.
\end{rmk}

Further, for a fixed trajectory $m_\cdot$ the solution of the initial value problem
\[
\frac{d}{dt}x(t) = f(x(t),m_t), \quad x(s) = z,
\]
evaluated at the time $r$, will be denoted by $\ds \traj{s}{z}{r}$.
\medskip

At the end of this section, we present some properties of the solution of the non-local continuity equation.

\begin{prp} \label{prp.bound.sol}
Let $T > 0$, $m_* \in \pp_p(\rd)$ and $\alpha>0$ be such that $\varsigma_p(m_*) \leq \alpha$, and $m_\cdot \in \Gamma_T(\pp_p(\rd))$ is the solution of the initial value problem~\eqref{eq.sys},~\eqref{eq.sys.data}.
Then for trajectory $\{ m_t:\ t \in [0,T] \}$ of the solution there exists a function $G_1(T,\alpha)$ such that, for all $t \in [0,T]$, the following inequality holds:
\[
\varsigma_p(m_t)
\leq G_1(T,\alpha).
\]
\end{prp}

\begin{prp} \label{prp.bound.f}
Let $T > 0$, $m_* \in \pp_p(\rd)$ and $\alpha>0$ be such that $\varsigma_p(m_*) \leq \alpha$, and $m_\cdot \in \Gamma_T(\pp_p(\rd))$ is the solution of the initial value problem~\eqref{eq.sys},~\eqref{eq.sys.data}.
Then there exists a function $G_1(T,\alpha)$ such that, for all $0 \leq s \leq r \leq T$ and for any $x \in \rd$, the following inequality holds:
\[
\left( \int_\rd \|f(\traj{s}{x}{r},m_r) - f(x,m_s)\|^p\ m_s(dx) \right)^{1/p} \leq  G_2(T,\alpha) \cdot (r-s).
\]
\end{prp}

The proofs of the above statements is provided in the appendix~\ref{app.traj}.
%======================================================================================
\subsection{Differentiability in the space of probabilistic measures} \label{sec.diff}
%======================================================================================
There are quite a numerous of various generalizations of the differentiability concept to the case of.
Some of them can be found in~\cite{Ambrosio_2005, Gangbo_2019, Cardaliaguet_2019}.
We use the concept of the intrinsic derivative proposed by P.-L. Lyons.
In addition, we will introduce the concepts of barycentric sub- and superdifferentials.
The relationship between some variants of generalizations of the concept of differential is discussed in more detail in Appendix~\ref{app.diff}.
\medskip

We follow~\cite{Cardaliaguet_2019} and use the concept of flat (variational) derivative to introduce the notion of intrinsic derivative of a functional defined on the space of probability measures.

\begin{dfn}[{\cite[Definition~2.2.1]{Cardaliaguet_2019}}]
Let $\Phi$ be a function from $\pp_p(\rd)$ to $\rr$.
A \dfnem{flat derivative} of the function $\Phi$ at the point $m_* \in \pp_p(\rd)$ is a function $\fld{\Phi}:\pp_p(\rd) \times \rd \to \rr$, that satisfy the equation
\[
\lim\limits_{s \downarrow 0} \frac{\Phi((1-s)m_* + sm) - \Phi(m_*)}s
= \int_\rd \fld{\Phi}(m_*,y)\ \Big[m(dy) - m_*(dy) \Big],
\]
for any measure $m \in \pp_p(\rd)$.
\end{dfn}

\begin{rmk}
If $\Phi:\pp_p(\rd) \to \rr$ has the flat derivative, then the following equality hold true:
\[
\Phi(m) - \Phi(m_*) = \int_0^1 \int_\rd \fld{\Phi}((1-s)m_*+sm,y)\ \Big[m(dy) - m_*(dy) \Big]ds
\]
for all $m_*,m \in \pp_p(\rd)$.
\end{rmk}

\begin{rmk}
Due to the fact that $\fld{\Phi}$ is defined up to an additive constant, its selection is based on the following normalization:
\[
\int_\rd \fld{\Phi}(m,y)\ m(dy) = 0.
\]
\end{rmk}

\begin{dfn}[{\cite[Definition~2.2.2]{Cardaliaguet_2019}}]
Let $\Phi:\pp_p(\rd) \to \rr$ has the flat derivative $\fld{\Phi}$, which is differentiable w.r.t. second variable.
Then an \dfnem{intrinsic derivative} of the function $\Phi$ at the point $m \in \pp_p(\rd)$ is a function $\dm \Phi:\pp_p(\rd) \times \rd \to \rds$ defined by the rule
\[
\dm \Phi(m,y) \triangleq \dy \fld{\Phi}(m,y).
\]
\end{dfn}

\begin{prp} \label{prp.flat.intr}
Let $m_*,m \in \pp_p(\rd)$, $\pi \in \Pi(m_*,m)$, $\Phi:\pp_p(\rd) \to \rr$ and let $\fld{\Phi}$ be differentiable w.r.t. second variable.
Then, the flat and intrinsic derivatives are related by the following relation:
\begin{equation} \label{eq.flat.intr}
\begin{split}
\int_\rd \frac{\delta \Phi}{\delta m}((1-s)m_*&+sm,y)\ \Big[m(dy) - m_*(dy) \Big]\\
&= \int\limits_{\rd\times\rd} \int_0^1 \dm \Phi((1-s)m_*+sm,y_*+q(y-y_*))\ dq \cdot (y-y_*)\ \pi(dy_*dy).
\end{split}
\end{equation}
\end{prp}
The proof of this proposition is given in Appendix~\ref{app.diff}.
\medskip

Finally, we introduce the concept of barycentric sub- and superdifferentials.
\begin{dfn} \label{dfn.bd}
Let $q = p' = \frac{p}{p-1}$ and let a functional $\phi: \pp_p(\rd) \to \rr$ be upper semi-continuous.
Then, a \dfnem{barycentric superdifferential} $\supd \phi(m)$ of the function $\phi$ at the point $m$ is a set of all functions $\gamma \in L_q(\rds,m;\rds)$ such that, for all $b \in L_p(\rd,m;\rd)$, there exists a function $\xi: \rr_+ \to \rr$ that satisfies the following conditions:
\begin{itemize}
\item $\xi(\tau) \to 0$ as $\tau \to 0$;
\item for any $\tau > 0$, one has that
\[
\phi((\Id + \tau b)\sharp m) - \phi(m)
\leq \tau \int_\rd \gamma(x) \cdot b(x)\ m(dx) + \tau \xi(\tau).
\]
\end{itemize}

A \dfnem{barycentric subdifferential} $\subd\phi(m)$ of the function $\phi$ at point $m$ is the set of all functions $\gamma \in L_q(\rds,m;\rds)$ such that
\[
(-\gamma) \in \supd(-\phi)(m).
\]

A \dfnem{barycentric differential} $\bd\phi(m)$ of the function $\phi$ at point $m$ is the intersection set $\supd\phi(m) \cap \subd\phi(m)$.
\end{dfn}

Everywhere below, except for Appendices, sub-/superdifferentiability will mean precisely barycentric sub-/superdifferentiability.

\begin{rmk}
If a function has a non-empty barycentric differential, then it contains exactly one element.
\end{rmk}

\begin{prp} \label{prp.intr.to.bar}
Let $m \in \pp_p(\rd)$ and a function $\phi: \pp_p(\rd) \to \rr$ has an intrinsic derivative $\dm \phi$ at point $m$.
Then $\dm \phi(m,\cdot) \in \bd\phi(m)$.
\end{prp}

The proof of this statement is given in the appendix~\ref{app.diff}.

The boundedness property of the barycentric superdifferential of a locally Lipschitz continuous function which, in our opinion, is also of interest is given in Appendix~\ref{app.bound}.
%======================================================================================
\section{Second Lyapunov method} \label{sec.2.lyap}
%======================================================================================
Let us define the basic notions of the Lyapunov stability theory.

\begin{dfn}
A measure $\hm$ is called an \dfnem{equilibrium of continuity equation~\eqref{eq.sys}} if the measure-valued function $m_\cdot$ defined by the rule $m_t \equiv \hm$ is a solution of continuity equation~\eqref{eq.sys}.
By Definition~\ref{dfn.sol.dist} we can reformulate this condition in the following form:
\[
\div (f(x,\hm) \hm) = 0,
\]
in the distribution sense.
\end{dfn}

\begin{dfn}
An equilibrium $\hm$ is called \dfnem{stable} provided that for every $\varepsilon>0$ there exists $\delta>0$ such that, if $m_* \in \pp_p(\rd)$, $W_p(\hm,m_*) < \delta$, $t>0$ and $m_\cdot$ is a solution of initial problem~\eqref{eq.sys},~\eqref{eq.sys.data}, then one has that
\[
W_p(\hm,m_t) < \varepsilon.
\]
\end{dfn}

\begin{dfn}
Let a function $\phi: \pp_p(\rd) \to \rr_+$ be locally Lipschitz continuous, $\phi(\hm)=0$ and satisfying the condition
\[
\inf\limits_{\mu \in \sall_\varepsilon(\hm)} \phi(\mu) > 0
\]
for some $R>0$ and for all $0 < \varepsilon \leq R$.

Function $\phi$ is called \dfnem{superdifferentiable Lyapunov function} for equilibrium $\hm$ if, in addition, it satisfies the following conditions
\begin{enumerate}
\item $\phi$ is superdifferentiable at the ball $\ds \ball_R(\hm)$;
\item $\ds \sup\limits_{\mu \in \ball_R(\hm)} \inf\limits_{\gamma \in \supd\phi(\mu)} \int_{\rd} \gamma(x) \cdot f(x,\mu)\ \mu(dx) \leq 0$.
\end{enumerate}

Function $\phi$ is called \dfnem{subdifferentiable Lyapunov function} for equilibrium $\hm$ if, in addition, it satisfies the following conditions
\begin{enumerate}
\item $\phi(\hm) = 0$;
\item $\inf\limits_{\mu \in \sall_\varepsilon(\hm)} \phi(\mu) > 0$ for all $0 < \varepsilon \leq R$;
\item $\phi$ is subdifferentiable at the ball $\ds \ball_R(\hm)$;
\item $\ds \sup\limits_{\mu \in \ball_R(\hm)} \inf\limits_{\gamma \in \subd\phi(\mu)} \int_{\rd} \gamma(x) \cdot f(x,\mu)\ \mu(dx) \leq 0$.
\end{enumerate}

Function $\phi$ is called \dfnem{differentiable Lyapunov function} for equilibrium $\hm$ if, in addition, it satisfies the following conditions
\begin{enumerate}
\item $\phi(\hm) = 0$;
\item $\inf\limits_{\mu \in \sall_\varepsilon(\hm)} \phi(\mu) > 0$ for all $0 < \varepsilon \leq R$;
\item $\phi$ is differentiable at the ball $\ds \ball_R(\hm)$;
\item $\ds \sup\limits_{\mu \in \ball_R(\hm)} \int_{\rd} \dm \phi(\mu,x) \cdot f(x,\mu)\ \mu(dx) \leq 0$.
\end{enumerate}
\end{dfn}

Now we formulate the main result of the section, namely, an analogue of the second Lyapunov method for stability of the equilibrium.
Recall that $\sall_\varepsilon(\hm)$ and $\ball_R(\hm)$ stands for the sphere of radius $\varepsilon$ and the closed ball of radius $R$ centered at $\hm$ respectively.

\begin{thm}[The second Lyapunov method for non-local continuity equation] \label{thm.second.method}
Let a measure $\hm \in \pp_p(\rd)$ be an equilibrium of continuity equation~\eqref{eq.sys} and let a function $f: \rd \times \pp_p(\rd) \to \rd$ satisfy Assumption~\ref{ass.lip}.
Assume additionally that there exists superdifferentiable Lyapunov function $\phi$ for the equilibrium $\hm$.
Then, equilibrium $\hm$ is stable.
\end{thm}

To prove this theorem, we need an auxiliary lemma stating that for the Lyapunov function $\phi$ the mapping $t\mapsto\phi(m_t)$ has a non-strict maximum at the point $t=0$.

\begin{lmm} \label{lemma.notinc}
Let $T>0$, a measure $m_* \in \pp_p(\rd)$, a measure-valued function $m_\cdot \in \Gamma_T(\pp_p(\rd))$ is a solution of the initial value problem~\eqref{eq.sys},~\eqref{eq.sys.data} on the interval $[0,T]$ and for locally Lipschitz continuous function $\phi: \pp_p(\rd) \to \rr_+$ the following conditions are hold:
\begin{enumerate}
\item $\phi$ is superdifferentiable at $m_t$ when $t \in [0,T]$;
\item $\ds \inf\limits_{\gamma \in \supd\phi(m_t)} \int_{\rd} \gamma(x) \cdot f(x,m_t)\ m_t(dx) \leq 0$ when $t \in [0,T]$.
\end{enumerate}
Then
\[
\phi(m_T) - \phi(m_0) \leq 0.
\]
\end{lmm}

\begin{proof}
Foremost, we claim the following fact: there exists a constant $C(T,m_*)$ such that, for each $\varepsilon > 0$ and for every $s \in [0,T)$, one can find $r \in (0,T]$ satisfying the inequality
\begin{equation} \label{eq.lemma.r.s}
\phi(m_r) - \phi(m_s)
\leq C(T,m_*) \cdot \varepsilon \cdot (r-s).
\end{equation}

To prove this fact we fix $\varepsilon > 0$ and the time $s \in [0,T)$.
From the second condition of the Lemma we can choose a supergradient $\gamma_s \in \supd(m_s)$ such that the following inequality holds:
\[
\int_{\rd} \gamma_s(x) \cdot f(x,m_s)\ m_s(dx) \leq \varepsilon.
\]
Further, let $r \in (s,T]$ and $\tau > 0$ be such that
\[
\begin{split}
\xi_s(\tau) &\leq \varepsilon, \\
\tau = (r-s) &\leq \varepsilon.
\end{split}
\]
Here $\xi_s$ is a function, defined in Definition~\ref{dfn.bd} that satisfies the inequality
\begin{equation} \label{eq.lemma.shs.s}
\phi\big((\Id + \tau b_s)\sharp m_s\big) - \phi(m_s)
\leq \tau \int_{\rd} \gamma_s(x) \cdot b_s(x)\ m_s(dx) + \tau \xi_s(\tau).
\end{equation}

Furthermore, let
\[
b_s(x) \triangleq f(x,m_s).
\]
We consider the difference
\[
\phi(m_r) - \phi(m_s)
= \Big[ \phi(m_r) - \phi\big((\Id + \tau b_s)\sharp m_s\big) \Big] + \Big[ \phi\big((\Id + \tau b_s)\sharp m_s\big) - \phi(m_s) \Big].
\]

Due to choice of $\tau$ we have that
\[
\phi\big((\Id + \tau b_s)\sharp m_s\big) - \phi(m_s)
\leq 2\varepsilon(r-s).
\]
\medskip

Now, let us evaluate the difference
\[
\phi(m_r) - \phi\big((\Id + \tau b_s)\sharp m_s\big).
\]
To this end, we put
\[
\beta_s(T,m_*) \triangleq G_1(T,\varsigma_p(m_*)) + T \cdot \|f(\cdot,m_s)\|_{L_p(\rd,m_s;\rd)}.
\]
Notice that, due to Proposition~\ref{prp.bound.mom},
\[
\varsigma_p((\Id + \tau b_s)\sharp m_s)
\leq G_1(T,\varsigma_p(m_*)) + \tau \|f(\cdot,m_s)\|_{L_p(\rd,m_s;\rd)}
\leq \beta_s(T,m_*),
\]
while Proposition~\ref{prp.bound.sol} yields the estimate
\[
\varsigma_p(m_r)
\leq G_1(T,\varsigma_p(m_*))
\leq \beta_s(T,m_*).
\]
Now, let us introduce a map
\[
\chi_1: z \mapsto \left( z + (r-s) \cdot b_s(z),\ z + \intsr f(\traj{s}{z}{t},m_t)\ dt \right),
\]
and define a plan by the rule
\[
\pi \triangleq \chi_1\sharp m_s.
\]
Then, thanks to the local Lipschitz continuity of the function $\phi$, Definition~\ref{dfn.w} and Proposition~\ref{prp.bound.f}, the following estimation is fulfilled:
\[
\begin{split}
\phi(m_r) &- \phi((\Id + \tau b_s)\sharp m_s)\\
&\leq K_{\beta_s(T,m_*)} \cdot W_p\big((\Id + \tau b_s)\sharp m_s,m_r\big)\\
&\leq K_{\beta_s(T,m_*)} \cdot \left( \int_\rd \left\| x + (r-s) \cdot b_s(x) - x - \intsr f(\traj{s}{x}{t},m_s)\ dt \right\|^p\ m_s(dx) \right)^{1/p}\\
&\leq K_{\beta_s(T,m_*)} \cdot G_2(T,\varsigma_p(m_*)) \cdot (r-s)^2
\leq K_{\beta_s(T,m_*)} \cdot G_2(T,\varsigma_p(m_*)) \cdot \varepsilon \cdot (r-s).
\end{split}
\]
Denoting $\hk \triangleq \sup\limits_{s \in [0,T]} K_{\beta_s(T,m_*)}$, we arrive to the inequality
\begin{equation} \label{eq.lemma.r.shs}
\phi(m_r) - \phi((\Id + \tau b_s)\sharp m_s)
\leq \hk \cdot G_2(T,\varsigma_p(m_*)) \cdot \varepsilon \cdot (r-s).
\end{equation}
\medskip

Letting $C(T,m_*) \triangleq 2 + \hk \cdot G_2(T,\varsigma_p(m_*))$ and using inequalities~\eqref{eq.lemma.shs.s} and~\eqref{eq.lemma.r.shs}, we obtain required estimate~\eqref{eq.lemma.r.s}.
\medskip

Next, we consider the set
\[
\Theta
\triangleq \{ \theta \in [0,T]: \phi(m_\theta) - \phi(m_0) \leq C(T,m_*) \cdot \varepsilon \cdot \theta \}.
\]
This set is non-empty.
Indeed, $0 \in \Theta$.
Denote $\hth \triangleq \sup \Theta$.
The continuity of the functions $m_\cdot$ and $\phi(\cdot)$ imply the inclusion $\hth \in \Theta$.
Let us show that $\hth = T$.

By contradiction, assume that $\hth < T$.
Due to~\eqref{eq.lemma.r.s} one can find $\theta \in (\hth,T]$ such that
\[
\phi(m_{\theta}) - \phi(m_{\hth}) \leq C(T,m_*) \cdot \varepsilon \cdot (\theta - \hth).
\]
By the construction of the constant $\hth$ we have that
\[
\phi(m_{\theta}) - \phi(m_0)
= (\phi(m_{\theta}) - \phi(m_{\hth})) - (\phi(m_{\hth}) - \phi(m_0))
\leq C(T,m_*) \cdot \varepsilon \cdot \theta.
\]
In other words, $\theta \in \Theta$ and $\theta > \hth = \sup \Theta$.
This contradicts with assumption that $\hth < T$.

Thus, $\sup \Theta = T \in \Theta$, and therefore
\[
\phi(m_T) - \phi(m_0) \leq C(T,m_*) \cdot \varepsilon \cdot T.
\]
Since we can choose $\varepsilon$ to be arbitrary small, the previous inequality implies
\[
\phi(m_T) - \phi(m_0) \leq 0.
\]
\end{proof}
\medskip

Now we can proof Theorem~\ref{thm.second.method}.

\begin{proof}[Proof of Theorem~\ref{thm.second.method}]
Due to continuity of the function $\phi$ at $\hm$, for every $\Omega>0$ there exists $\omega > 0$ such that, for each measure $m \in \pp_p(\rd)$ satisfying the condition $W_p(m,\hm) < \omega$, the following inequality holds:
\[
|\phi(m) - \phi(\hm)| < \Omega.
\]
From the assumptions of the Theorem, it follows that
\[
|\phi(m) - \phi(\hm)|
= |\phi(m)|
= \phi(m).
\]

Furthermore, to show the stability of the equilibrium $\hm$, let us assume the opposite.
Namely, let there exist $\varepsilon > 0$ such that, for every $\delta > 0$, one can find a measure $m_*\in \pp_p(\rd)$ and a time $\hatt>0$ such that
\[
\begin{split}
W_p(\hm,m_*) &< \delta,\\
W_p(\hm,m_\hatt) &= \varepsilon.
\end{split}
\]
Here, the measure-valued function $m_\cdot$ is a solution of the initial value problem~\eqref{eq.sys},~\eqref{eq.sys.data}.

Without loss of generality, we can assume that
\begin{itemize}
\item $\ds \varepsilon \leq R$;
\item $\ds \Omega = \inf \big\{ \phi(\mu):\ W_p(\mu,\hm) = \varepsilon \big\}$;
\item $\ds \delta \leq \omega \leq R$;
\item $\ds \hatt$ is minimal.
\end{itemize}
Due to choice of $\delta$, $\Omega$ and $\omega$, we have that $\phi(m_*) < \hphi$.
Since $m_\cdot$ is a solution of continuity equation~\eqref{eq.sys} and
\[
m_t \in \ball_\varepsilon(\hm) \subseteq \ball_R(\hm)
\]
when $t \in [0,\hatt]$, then the definition of the function $\hphi$ and Lemma~\ref{lemma.notinc} imply the following inequality:
\[
\hphi \leq \phi(m_\hatt) \leq \phi(m_*) < \hphi.
\]
This contradiction proves the theorem.
\end{proof}

Similarly, the Theorem on the second Lyapunov method can be formulated for the case of subdifferentiable Lyapunov function.

\begin{thm}
Let a measure $\hm \in \pp_p(\rd)$ be an equilibrium of continuity equation~\eqref{eq.sys} and let a function $f: \rd \times \pp_p(\rd) \to \rd$ satisfy Assumption~\ref{ass.lip}.
Assume additionally that there exists subdifferentiable Lyapunov function $\phi$ for the equilibrium $\hm$.
Then, equilibrium $\hm$ is stable.
\end{thm}

As in the previous case, the proof of the Theorem relies on the following lemma.

\begin{lmm}
Let $T>0$, a measure $m_* \in \pp_p(\rd)$, a measure-valued function $m_\cdot \in \Gamma_T(\pp_p(\rd))$ is a solution of the initial value problem~\eqref{eq.sys},~\eqref{eq.sys.data} on the interval $[0,T]$ and for locally Lipschitz continuous function $\phi: \pp_p(\rd) \to \rr_+$ the following conditions are hold:
\begin{enumerate}
\item $\phi$ is subdifferentiable at $m_t$ when $t \in [0,T]$;
\item $\ds \inf\limits_{\gamma \in \subd\phi(m_t)} \int_{\rd} \gamma(x) \cdot f(x,m_t)\ m_t(dx) \leq 0$ when $t \in [0,T]$.
\end{enumerate}
Then
\[
\phi(m_T) - \phi(m_0) \leq 0.
\]
\end{lmm}

Due to Proposition~\ref{prp.intr.to.bar} and Theorem~\ref{thm.second.method} the following result can be directly derived.

\begin{cor}
Let a measure $\hm \in \pp_p(\rd)$ be an equilibrium of continuity equation~\eqref{eq.sys} and let a function $f: \rd \times \pp_p(\rd) \to \rd$ satisfy Assumption~\ref{ass.lip}.
Assume additionally that there exists differentiable Lyapunov function $\phi$ for the equilibrium $\hm$.
Then, equilibrium $\hm$ is stable.
\end{cor}
%======================================================================================
\subsection{Example: Gradient flow}\label{sec.2.ex}
%======================================================================================
Let $v: \rd \times \pp_2(\rd) \to \rr_+$.
Consider a functional
\[
\phi(m) = \int_\rd v(x,m)\ m(dx).
\]
We assume that
\begin{itemize}
\item $\dx v$ and $\dm v$ are defined on $\rd \times \pp_2(\rd)$ and $\rd \times \pp_2(\rd) \times \rd$ respectively;
\item $v(x,m)$ is Lipschitz continuous with the constant $l>0$;
\item $\dx v(x,m)$ and $\int_\rd \dm v(y,m,x)\ m(dy)$ are Lipschitz continuous;
\item $v(0,\delta_0) = 0$;
\item $\dx v(0,\delta_0) + \dm v(0,\delta_0,0) = 0$;
\item for every $\varepsilon > 0$, one can find a constant $\alpha_\varepsilon > 0$ such that, for all $m \in \sall_\varepsilon(\delta_0)$, the inequality $v(\cdot,m) \geq \alpha_\varepsilon > 0$ holds.
\end{itemize}

Let us consider a system of particles driven by the dynamics
\[
\dot{x} = - \big( \dm \phi(m,x) \big)^\tran.
\]
The corresponding continuity equation takes the form
\begin{equation} \label{eq.2.ex.cont}
\partial_t m_t - \div \big(\big( \dm \phi(m,x) \big)^\tran m_t \big) = 0.
\end{equation}
Notice, that (see~\cite[Proposition~A.3]{Averboukh_2022_arxiv})
\[
\dm \phi(m,x) = \dx v(x,m) + \int_\rd \dm v(y,m,x)\ m(dy).
\]

Let us show that the measure $\hm = \delta_0$ is equilibrium of equation~\eqref{eq.2.ex.cont}, i.e.
\[
\int_\rd \dx \varphi(x) \cdot \dm \phi(\hm,x)^\tran\ \hm(dx)
= \dx \varphi(0) \cdot \dm \phi(\hm,0)^\tran
= 0,
\]
for all $\varphi \in \cc(\rd)$.
The left-hand side of this equality can be rewritten as
\[
\begin{split}
\int_\rd &\dx \varphi(x) \cdot \left( \dx v(x,\hm) + \int_\rd \dm v(y,\hm,x)\ \hm(dy) \right)^\tran\ \hm(dx)\\
&\hspace{5cm}= \dx \varphi(0) \cdot \left( \dx v(0,\delta_0) + \dm v(0,\delta_0,0) \right)^\tran
= 0.
\end{split}
\]

Now, check the conditions of the Theorem~\ref{thm.second.method}.
Due to Proposition~\ref{prp.intr.to.bar}
\[
\dm \phi(m,\cdot) \in \bd\phi(m) \subseteq \supd\phi(m).
\]
Thus, functional $\phi$ has the following properties:
\begin{itemize}
\item $\phi(\hm) = 0$;
\item $\inf\limits_{m \in \sall_\varepsilon(\hm)}\phi(m) > 0$ for all $\varepsilon > 0$;
\item since
\[
\begin{split}
\phi(\mu) - \phi(\nu)
&= \int_\rd v(x,\mu)\ \mu(dx) - \int_\rd v(y,\nu)\ \nu(dy) \\
&\hspace{1cm}= \int\limits_{\rd\times\rd} (v(x,\mu)-v(y,\nu))\ \pi(dxdy) \\
&\hspace{1.5cm}\leq \int\limits_{\rd\times\rd} l \cdot \big(\|x-y\| + W_p(\mu,\nu)\big)\ \pi(dxdy) \\
&\hspace{2cm}\leq l \left( \ \int\limits_{\rd\times\rd} \|x-y\|^p\ \pi(dxdy) \right)^{1/p} + l \cdot W_p(\mu,\nu) \leq 2l \cdot W_p(\mu,\nu),
\end{split}
\]
where $\pi \in \Pi_o(\mu,\nu)$, the function $\phi$ is globally Lipschitz continuous;
\item the following inequality holds:
\[
\int_\rd -\dm \phi(m,x) \cdot \dm \phi(m,x)^\tran\ m(dx)
= - \|\dm \phi(m,\cdot)\|_{L_2(\rd,m;\rds)}^2 \leq 0.
\]
\end{itemize}

Therefore, the equilibrium $\hm$ and the function $\phi$ satisfy the conditions of Theorem~\ref{thm.second.method}.
This yields the stability of $\hm$.
%======================================================================================
\section{Stability of systems with a linear vector field} \label{sec.1.lyap}
%======================================================================================
In this section, we will assume that $p=2$.
Further, we will need additional assumptions on the dynamics~$f$ and the equilibrium $\hm$ given below.

\begin{ass} \label{ass.b.dif}
Function $f(x,m)$ is linear and has separated variables, i.e.,
\[
f(x,m) = Ax + \int_\rd By\ m(dy),
\]
where $A$ and $B$ are constant matrices with dimension $d \times d$.
\end{ass}

\begin{ass} \label{ass.ac}
Measure $\hm$ is absolutely continuous w.r.t. the Lebesgue measure.
\end{ass}

Assumption~\ref{as.ac} and \cite[Theorem~6.2.4]{Ambrosio_2005} imply the existence of the unique optimal transport $P$ between $\hm$ and each measure $m\in\pp_2(\rd)$.
Hence, there exists an optimal plan of the form $\pi = (\Id,P) \sharp\hm\in \Pi_o(\hm,m)$.

As a candidate Lyapunov function, we will consider the half of the squared distance to the equilibrium:
\[
\phi(m) \triangleq \frac12 W_2^2(m,\hm).
\]
First, let us find the barycentric superdifferential of the function $\phi$.
\begin{prp} \label{prp.w22.supd}
Let $\hm,m \in \pp_2(\rd)$ and $\pi \in \Pi_o(\hm,m)$.
Then, the function
\[
\hpi^1(x) \triangleq \int_\rd (x-\hx)^\tran\ \pi(d\hx|x)
\]
is a barycentric superdifferential of the function $\phi$ at $m$.
\end{prp}
\begin{proof}
From \cite[Theorem~10.2.2]{Ambrosio_2005} we have that, for all measures $\mu^1,\mu^2,\mu^3 \in \pp_2(\rd)$ and each measure $\mbmu \in \pp(\rd \times \rd \times \rd)$ such that
\[
\begin{split}
\proj^{1,2}\sharp\mbmu &\in \Pi_o(\mu^1,\mu^2),\\
\proj^{3}\sharp\mbmu &= \mu^3,
\end{split}
\]
the following inequality holds true:
\begin{equation} \label{eq.supd.w2.orig}
\begin{split}
\phi(\mu^3)-\phi(\mu^1)
&\leq \int\limits_{(\rd)^3} (x_1-x_2)^\tran \cdot (x_3-x_1)\ \mbmu(dx_1 dx_2 dx_3)\\
&\hspace{3cm}+ \int\limits_{(\rd)^3} \|x_3 - x_1\|^2\ \mbmu(dx_1 dx_2 dx_3).
\end{split}
\end{equation}
Further, we choose an arbitrary $\tau > 0$ and $b \in L_2(\rd,\hm;\rd)$, and measures $\mu^1,\mu^2,\mu^3 \in \pp_2(\rd)$ such that
\[
\begin{split}
\mu^1 &= m,\\
\mu^2 &= \hm,\\
\mu^3 &= (\Id + \tau b) \sharp m,\\
\mbmu &= (\proj^2,\proj^1,\proj^2 + \tau b \circ \proj^2) \sharp \pi.
\end{split}
\]
Renaming variables:
\[
\begin{split}
x &= x_1,\\
\hx &= x_2
\end{split}
\]
and using definition of the push-forward measure, we rewrite the right-hand side of inequality~\eqref{eq.supd.w2.orig} can be rewritten in the form:
\[
\begin{split}
\int\limits_{(\rd)^2} (x-\hx)^\tran \cdot \tau b(x)\ &\pi(d\hx dx)
+ \int\limits_\rd \|\tau b(x)\|^2\ \pi(d\hx dx)\\
&= \int\limits_{(\rd)^2} (x-\hx)^\tran \cdot \tau b(x)\ \pi(d\hx dx)
+ \tau^2 \cdot \|b\|_{L_2(\rd,m;\rd)}^2.
\end{split}
\]
Disintegrating the measure $\pi$ w.r.t. the variable $x$, we obtain
\[
\begin{split}
\int_\rd \int_\rd (x-\hx)^\tran \cdot \tau b(x)\ &\pi(d\hx|x)\ m(dx)
+ \tau^2 \cdot \|b\|_{L_2(\rd,m;\rd)}^2\\
= \int\limits_\rd &\left( \int_\rd (x-\hx)^\tran\ \pi(d\hx|x) \cdot \tau b(x) \right)\ m(dx)
+ \tau^2 \cdot \|b\|_{L_2(\rd,m;\rd)}^2.
\end{split}
\]
Let us denote
\[
\xi(\tau) \triangleq \tau \|b\|_{L_2(\rd,m;\rd)}^2.
\]
Therefore,
\[
\phi((\Id + \tau b) \sharp m) - \phi(m)
\leq \tau \int_\rd \hpi^1(x) \cdot b(x)\ m(dx)
+ \tau \cdot \xi(\tau).
\]
\end{proof}

Now, we can formulate the main result of the section, namely, an analogue of the Lyapunov method for stability of the equilibrium for a system with a linear vector field.

\begin{thm}[Stability of a system with a linear vector field] \label{thm.first.method}
Let a function $f: \rd \times \pp_2(\rd)$ satisfy Assumptions~\ref{ass.lip} and~\ref{ass.b.dif}, and let a measure $\hm \in \pp_2(\rd)$ be an equilibrium that satisfies Assumption~\ref{ass.ac}.
Additionally, assume that the following condition holds for every $\xi \in \Tan(\hm)\setminus\{0\}$:
\begin{equation} \label{eq.neq.qform}
\int_\rd \xi(\hx)^\tran \cdot A \cdot \xi(\hx)\ \hm(d\hx) + \int_\rd \xi(\hx)^\tran\ \hm(d\hx) \cdot B \cdot \int_\rd \xi(\hy)\ \hm(d\hy)
\leq 0.
\end{equation}
Then, the equilibrium $\hm$ is stable.
\end{thm}

To prove this theorem, we need two auxiliary results.

\begin{lmm} \label{lemma.diff.eq}
Let a measure $\hm \in \pp_2(\rd)$ satisfy Assumption~\ref{ass.ac}.
Then, for every measure $m \in \pp_2(\rd)$ and every Borel function $f: \rd \times \pp_2(\rd) \to \rd$, the following condition holds:
\begin{equation} \label{eq.diff.eq}
\int\limits_{\rd\times\rd} (x-\hx)^\tran \cdot f(x,m)\ \pi(d\hx dx)
= \int\limits_{\rd\times\rd} (x-\hx)^\tran \cdot \big(f(x,m)-f(\hx,\hm)\big)\ \pi(d\hx dx),
\end{equation}
where $\pi = (\Id,P)\sharp \hm \in \Pi_o(\hm,m)$.
\end{lmm}

\begin{proof}
We add and subtract $\int_{\rd\times\rd} (x-\hx)^\tran \cdot f(\hx,\hm)\ \pi(d\hx dx)$ in the left-hand side of~\eqref{eq.diff.eq}.
Thus,
\begin{equation} \label{eq.diff.eq.left}
\begin{split}
\int\limits_{\rd\times\rd} (x-\hx)^\tran \cdot f(x,m)\ \pi(d\hx dx)
&= \int\limits_{\rd\times\rd} (x-\hx)^\tran \cdot \big(f(x,m)-f(\hx,\hm)\big)\ \pi(d\hx dx)\\
&\hspace{1cm}+ \int\limits_{\rd\times\rd} (x-\hx)^\tran \cdot f(\hx,\hm)\ \pi(d\hx dx).
\end{split}
\end{equation}
Now, let us consider the second term.
Due to the definition of the plan $\pi$, we have
\[
\int\limits_{\rd\times\rd} (x-\hx)^\tran \cdot f(\hx,\hm)\ \pi(d\hx dx)
= \int\limits_\rd (P(\hx)-\hx)^\tran \cdot f(\hx,\hm)\ \hm(d\hx).
\]
It follows from~\cite[Proposition~8.5.2]{Ambrosio_2005} that
\begin{equation} \label{eq.p.id.tan}
(P-\Id) \in \Tan(\hm),
\end{equation}
i.e., there exists a sequence $\varphi_n \in \cc(\rd)$ such that
\[
\nabla\varphi_n \to (P-\Id) \text{ in the } L_2(\rd,\hm;\rd) \text{ sense, as } n \to \infty.
\]
At the same time, due to the fact that $\hm$ is an equilibrium one has that
\[
\int_\rd \nabla\varphi_n(\hx) \cdot f(\hx,\hm)\ \hm(d\hx) = 0.
\]
Passing to the limits, we obtain
\[
\int_\rd (P(\hx)-\hx)^\tran \cdot f(\hx,\hm)\ \hm(d\hx) = 0.
\]
Thus,
\[
\int\limits_{\rd\times\rd} (x-\hx)^\tran \cdot f(\hx,\hm)\ \pi(d\hx dx) = 0.
\]
Together with~\eqref{eq.diff.eq.left}, it implies the conclusion of the Lemma.
\end{proof}

\begin{lmm} \label{lemma.neg.diff}
Let a function $f: \rd \times \pp_2(\rd) \to \rd$ satisfy Assumptions~\ref{ass.lip} and~\ref{ass.b.dif}, a measure $\hm \in \pp_2(\rd)$ satisfy Assumption~\ref{ass.ac},
and let the condition~\eqref{eq.neq.qform} holds for all $\xi \in \Tan(\hm)\setminus\{0\}$.
Then, for all $m \in \pp_2(\rd)$, one has that
\begin{equation} \label{eq.neg.diff}
\int\limits_{\rd\times\rd} (x-\hx)^\tran \cdot \big(f(x,m)-f(\hx,\hm)\big)\ \pi(d\hx dx) \leq 0,
\end{equation}
where $\pi = (\Id,P)\sharp \hm \in \Pi_o(\hm,m)$.
\end{lmm}
\begin{proof}
Everywhere below, integration is understood in a component-by-component sense.
Let us consider the integral expression of the left-hand side of the inequality~\eqref{eq.diff.eq}.
\begin{equation} \label{eq.orig}
(x-\hx)^\tran \cdot \big(f(x,m)-f(\hx,\hm)\big)
= (x-\hx)^\tran \cdot  \big(f(x,m)-f(\hx,m)\big)
+ (x-\hx)^\tran \cdot  \big(f(\hx,m)-f(\hx,\hm)\big).
\end{equation}
We will analyze the terms of the right-hand of this equality side separately.
The first term due to the Lagrange's theorem takes the form
\[
\begin{split}
(x-\hx)^\tran \cdot \big(f(x,m)-f(\hx,m)\big)
&= (x-\hx)^\tran \cdot \int_0^1 \dx f(\hx + q (x-\hx),m)\ dq \cdot (x-\hx)\\
&= (x-\hx)^\tran \cdot A \cdot (x-\hx).
\end{split}
\]
Proposition~\ref{prp.flat.intr} implies that second term can be rewritten as
\[
\begin{split}
(x-\hx) &\cdot \big(f(\hx,m)-f(\hx,\hm)\big) \\
&= (x-\hx)^\tran \cdot \int_0^1 \int\limits_{\rd\times\rd} \int_0^1 \dm f(\hx,(1-s)\hm+sm,\hy+q(y-\hy))\ dq \\
&\hspace{9cm}\cdot (y-\hy)\ \pi(d\hy dy)\ ds \\
&= \int\limits_{\rd\times\rd} (x-\hx)^\tran \cdot B \cdot (y-\hy)\ \pi(d\hy dy).
\end{split}
\]
The last transition is valid by virtue of~\cite[Proposition~A.2]{Averboukh_2022_arxiv}.

Thus, integrating equality~\eqref{eq.orig} w.r.t. the measure $\pi(d\hx dx)$, we obtain
\[
\begin{split}
\int\limits_{\rd\times\rd} &(x-\hx)^\tran \cdot \big(f(x,m)-f(\hx,\hm)\big)\ \pi(d\hx dx)\\
&= \int\limits_{\rd\times\rd} (x-\hx)^\tran \cdot A \cdot (x-\hx)\ \pi(d\hx dx)\\
&\hspace{1cm}+ \int_\rd (x-\hx)^\tran\ \pi(d\hx dx) \cdot B \cdot \int_\rd (y-\hy)\ \pi(d\hy dy).
\end{split}
\]
Using the form of the plan $\pi$ and optimality of the transport $P$, we obtain the following equation:
\[
\begin{split}
\int\limits_{\rd\times\rd} &(x-\hx)^\tran \cdot  \big(f(x,m)-f(\hx,\hm)\big)\ \pi(d\hx dx)\\
&= \int_\rd (P(\hx)-\hx)^\tran \cdot A \cdot (P(\hx)-\hx)\ \hm(d\hx)\\
&\hspace{1cm}+ \int_\rd (P(\hx)-\hx)^\tran\ \hm(d\hx) \cdot B \cdot \int_\rd (P(\hy)-\hy)\ \hm(d\hy).
\end{split}
\]
Due to~\eqref{eq.neq.qform} and~\eqref{eq.p.id.tan} right-hand side of~\eqref{eq.orig} is non-positive.
\end{proof}

Now we give the proof of the theorem.

\begin{proof}[Proof of the Theorem~\ref{thm.first.method}]
Let us check the conditions of Theorem~\ref{thm.second.method} and show that the function
\[
\phi(m) = \frac12 W_2(m,\hm)
\]
is a superdifferentiable Lyapunov function.
Due to the assumptions of Theorem it is suffices to check only the condition for the superdifferential of the function $\phi$.

Lemma~\ref{prp.w22.supd} states that $\supd\phi(m)$ contains a function $\hpi^1(\cdot)$ where $\pi \in \Pi_o(m,\hm)$.
Let us consider the integral
\[
\int_\rd \hpi^1(x) \cdot f(x,m)\ m(dx)
\]
where $m \in \pp_2(\rd)$.
Using the definition of $\hpi^1$, we have that
\[
\begin{split}
\int\limits_\rd \hpi^1(x) \cdot f(x,m)\ m(dx)
&= \int\limits_\rd \left( \int_\rd (x-\hx)^\tran\ \pi(d\hx|x) \cdot f(x,m) \right)\ m(dx)\\
&= \int_\rd \int_\rd (x-\hx)^\tran \cdot f(x,m)\ \pi(d\hx|x)\ m(dx)\\
&= \int\limits_{\rd\times\rd} (x-\hx)^\tran \cdot f(x,m)\ \pi(d\hx dx).
\end{split}
\]
Lemmas~\ref{lemma.diff.eq} and~\ref{lemma.neg.diff} imply
\[
\int\limits_\rd \hpi^1(x) \cdot f(x,m)\ m(dx)
= \int\limits_{\rd\times\rd} (x-\hx)^\tran \cdot f(x,m)-f(\hx,\hm)\ \pi(d\hx dx) \leq 0.
\]
Thus, due to Theorem~\ref{thm.second.method}, the equilibrium $\hm$ is stable.
\end{proof}
%======================================================================================
\subsection{Example: A mean field system of pendulums}
%======================================================================================
Let $d=2n$ and $x_1,x_2 \in \rn$.
Denote
\begin{itemize}
\item $\ds x = \begin{pmatrix}
x_1\\
x_2
\end{pmatrix}$;
\item $\ds H(x) = \frac{x_1^2}2 + \frac{x_2^2}2$;
\item $\ds H_1 = H_{x_1}$, $H_2 = H_{x_2}$;
\item $\ds A = \begin{pmatrix}
\mathbb{O} & \mathbb{I}\\
-\mathbb{I} & \mathbb{O}\end{pmatrix}$,
\end{itemize}
where $\mathbb{O}$ and $\mathbb{I}$ are zero and identical matrices with dimension $n \times n$ respectively.
Notice that
\[
Ax = \begin{pmatrix}
H_1(x)\\
-H_2(x)
\end{pmatrix}
= \begin{pmatrix}
x_2\\
-x_1
\end{pmatrix}.
\]
If $n > 1$, then a vector composed of multidimensional quantities is interpreted as a vertical concatenation, squaring as a scalar square, $H_1$ (respectively $H_2$) denotes the vector of partial derivatives of $H$ w.r.t. the components of $x_1$ (respectively $x_2$).

Consider a system
\begin{equation} \label{eq.ex.osc}
\dot{x} = Ax + \int_\rd By\ m(dy) \triangleq f(x,m),
\end{equation}
where $B$ is a negative-definite matrix.
Notice that first term in this equation describes the free movement of a pendulum, while the second one shows an influence of all other pendulums on the given one.
Thus, regard dynamics~\eqref{eq.ex.osc} as a mean field system of pendulums.
The continuity equation corresponding to this system takes the form:
\begin{equation} \label{eq.1.ex.cont}
\partial_t m_t + \div \left( \left(Ax + \int_\rd By\ m_t(dy)\right) m_t \right)
= 0.
\end{equation}

Let $\hm$ be a Gibbs measure generated by the Hamiltonian $H(x)$.
Recall, that it has a density
\[
g(x) \triangleq \frac1{Z(\beta)}e^{-\beta H(x)},
\]
where $Z(\beta)$ stands for the normalizing coefficient, and that its first moment is equal to zero.
Now, we show that the measure $\hm$ is an equilibrium for equation~\eqref{eq.1.ex.cont}, i.e.,
\[
\int_\rd \dx \varphi(x) \cdot \left( Ax + \int_\rd By\ \hm(dy) \right)\ \hm(dx)
= 0.
\]
Due to the described properties of the Gibbs measure, it suffices to prove that
\begin{equation} \label{eq.ex.hibbs}
\int_\rd \dx (\varphi(x) \cdot Ax) \cdot g(x)\ dx
= 0.
\end{equation}
Integrating the left-hand side by parts, we arrive to the equality
\[
\int_\rd (\dx \varphi(x) \cdot Ax) \cdot g(x)\ dx
= \int_\rd \divx \big( Ax \cdot g(x) \big) \varphi(x)\ dx.
\]
The Gibbs measure is invariant w.r.t. the Hamiltonian dynamics.
Therefore, its density satisfies the Liouville equation (see, for example, ~\cite{Kozlov_2002}):
\[
\divx \big( Ax \cdot g(x) \big) = 0.
\]
This implies~\eqref{eq.ex.hibbs}.

Let us check the conditions of Theorem~\ref{thm.first.method}.
For an arbitrary nonzero integrable function $\xi : \rd \to \rd$, we have that
\[
\begin{split}
\int_\rd \xi(\hx)^\tran &\cdot A \cdot \xi(\hx)\ \hm(d\hx)\\
&\hspace{1cm}+ \int_\rd \xi(\hx)^\tran\ \hm(d\hx) \cdot B \cdot \int_\rd \xi(\hy)\ \hm(d\hy)\\
&=\int_\rd \xi(\hx)^\tran\ \hm(d\hx) \cdot B \cdot \int_\rd \xi(\hy)\ \hm(d\hy).
\end{split}
\]
Due to the negative-definiteness of the matrix $B$, this expression is negative.

Hence, the equilibrium $\hm$ and the function $f$ satisfy the conditions of Theorem~\ref{thm.first.method}.
This yields the stability of $\hm$.
%======================================================================================
\section{Conclusion}
%======================================================================================
In this paper, we consider an autonomous dynamic system in the space of probabilistic measures defined by a non-local continuity equation.
Sufficient conditions for the local stability based on the second Lyapunov method have been studied.
The above results can be extended to such issues of significant interest as trajectory stability analysis and global stability problems.
%======================================================================================
\appendix
%======================================================================================
\section{Properties of solutions of non-local continuity equation} \label{app.traj}
%======================================================================================
\begin{proof}[Proof of the Proposition~\ref{prp.bound.sol}]
By the definition of the solution of system~\eqref{eq.sys.x}, we have that
\[
\traj{0}{x_*}{t}
= x_*
+ \int\limits_0^t f(\traj{0}{x_*}{\tau},m_\tau)\ d\tau.
\]
Evaluating the $L_p(\rd,m_*;\rd)$-norm of both sides and applying the Minkowski inequality to right-hand side twice, we arrive to the estimate
\begin{equation} \label{eq.bound.sol}
\begin{split}
\left( \int_\rd \|\traj{0}{x_*}{t}\|^p\ m_*(dx_*) \right)^{1/p}
&\leq \left( \int_\rd \|x_*\|^p\ m_*(dx_*) \right)^{1/p}\\
&\hspace{1cm}+ \int\limits_0^t \left( \int_\rd \|f(\traj{0}{x_*}{\tau},m_\tau)\|^p\ m_*(dx_*)\right)^{1/p} d\tau.
\end{split}
\end{equation}
Furthermore, Remark~\ref{ass.ulin} and the triangle inequality, we have that
\[
\begin{split}
\left( \int_\rd \|f(\traj{0}{x_*}{\tau},m_\tau)\|^p\ m_*(dx_*)\right)^{1/p}
&\leq \left( \int_\rd C_1^p(1 + \|\traj{0}{x_*}{\tau}\| + \varsigma_p(m_\tau))^p\ m_*(dx_*)\right)^{1/p}\\
&\leq C_1 \left(1 + \left( \int_\rd \|\traj{0}{x_*}{\tau}\|^p\ m_*(dx_*) \right)^{1/p} + \varsigma_p(m_\tau) \right).
\end{split}
\]

Now let us introduce a map
\[
\chi_2: z \mapsto \traj{0}{z}{s}.
\]
Then, using \cite[Proposition~8.1.8]{Ambrosio_2005}, we conclude that
\[
\begin{split}
\left( \int_\rd \|\traj{0}{x_*}{s}\|^p\ m_*(dx_*) \right)^{1/p}
&= \left( \int_\rd \|x_*\|^p\ \big(\chi_2\sharp m_*\big)(dx_*) \right)^{1/p}\\
&= \left( \int_\rd \|x_*\|^p\ m_s(dx_*) \right)^{1/p}
= \varsigma_p(m_s)
\end{split}
\]
for each time $s$.

Hence, inequality~\eqref{eq.bound.sol} leads the following one:
\[
\varsigma_p(m_t)
\leq \varsigma_p(m_*)
+ C_1 \cdot t + \int\limits_0^t 2 C_1 \cdot \varsigma_p(m_\tau)\ d\tau.
\]
Finally, applying the Gronwall-Bellman inequality in the integral form and using the conditions of the Proposition imposed on $\alpha$ and $T$, we obtain
\[
\varsigma_p(m_t) \leq (C_1 \cdot T + \alpha) \cdot e^{2 C_1 \cdot T}.
\]
The last estimate depends only on the constant $\alpha$ and the final time $T$.
\end{proof}

\begin{prp} \label{prp.bound.traj}
Let $T > 0$, $m_*$ and $\alpha$ be such that $\varsigma_p(m_*) \leq \alpha$ and let $m_\cdot \in \Gamma_T(\pp_p(\rd))$ be a solution of the initial value problem~\eqref{eq.sys},~\eqref{eq.sys.data}.
Then, there exists a function $G_3(T)$ such that, for all $0 \leq s \leq r \leq T$ and for each $z \in \rd$, the following inequality holds:
\[
\|\traj{s}{z}{r} - z\|
\leq G_3(T) \cdot (1 + \|z\| + G_1(T,\alpha)) \cdot (r-s).
\]
\end{prp}
\begin{proof}
By the definition of the solution, we have
\[
\traj{s}{z}{r}
= z + \intsr f(\traj{s}{z}{\tau},m_\tau) d\tau.
\]
Remark~\ref{ass.ulin} gives that
\[
\begin{split}
\|\traj{s}{z}{r} - z\|
&= \left\| \intsr f(\traj{s}{z}{\tau},m_\tau) d\tau \right\|\\
&\leq \intsr \| f(\traj{s}{z}{\tau},m_\tau)\| d\tau
\leq C_1 \intsr (1 + \|\traj{s}{z}{\tau}\| + \varsigma(m_\tau)) d\tau.
\end{split}
\]
Now we add and subtract $z$ to $\traj{s}{z}{\tau}$ at the right-hand side.
Furthermore, using the triangle inequality and Proposition~\ref{prp.bound.sol}, we obtain
\[
\begin{split}
\|\traj{s}{z}{r} - z\|
&\leq C_1 \intsr (1 + \|z\| + \varsigma(m_\tau)) d\tau + C_1 \intsr \|\traj{s}{z}{\tau}-z\| d\tau\\
&\leq C_1 (1 + \|z\| + G_1(T,\alpha)) \cdot (r-s) + C_1 \intsr \|\traj{s}{z}{\tau}-z\| d\tau.
\end{split}
\]
Applying the Gronwall-Bellman in the integral form, we arrive to the estimate
\[
\begin{split}
\|\traj{s}{z}{r} - z\|
&\leq C_1 (1 + \|z\| + G_1(T,\alpha)) \cdot (r-s) e^{\intsr C_1\ d\tau}\\
&\leq C_1 \cdot e^{C_1 \cdot T} \cdot (1 + \|z\| + G_1(T,\alpha)) \cdot (r-s).
\end{split}
\]
The first and second terms here depend only on the final time~$T$.
\end{proof}

\begin{prp} \label{prp.bound.w2sol}
Let $T > 0$, $m_*$ and $\alpha$ be such that $\varsigma_p(m_*) \leq \alpha$ and let $m_\cdot \in \Gamma_T(\pp_p(\rd))$ be a solution of the initial value problem~\eqref{eq.sys},~\eqref{eq.sys.data}.
Then, there exists a function $G_4(T,\alpha)$ such that, for all $0 \leq s \leq r \leq T$, the following inequality holds:
\[
W_p(m_s,m_r) \leq  G_4(T,\alpha) \cdot (r-s).
\]
\end{prp}
\begin{proof}
By the definition
\[
W_p(m_s,m_r)
= \left( \inf\limits_{\pi \in \Pi(m_s,m_r)} \int_\rd \|x - y\|^p\ \pi(dxdy) \right)^{1/p}
\]
Now, let us introduce a map
\[
\chi_3: z \mapsto (z,\ \traj{s}{z}{r})
\]
and choose a plan defined by the rule
\[
\pi = \chi_3\sharp m_s.
\]
Then,
\[
W_p(m_s,m_r)
\leq \left( \int_\rd \|x - y\|^p\ \big( \chi_3\sharp m_s \big)(dxdy) \right)^{1/p}
= \left( \int_\rd \|z - \traj{s}{z}{r}\|^p\ m_s(dz) \right)^{1/p}.
\]
Due to Propositions~\ref{prp.bound.sol},~\ref{prp.bound.traj} and Minkowski inequality, we have that
\[
W_p(m_s,m_r)
\leq G_3(T) \cdot (r-s) \cdot (1 + \varsigma_p(m_s) + G_1(T,\alpha))
\leq G_3(T) \cdot (r-s) \cdot (1 + 2G_1(T,\alpha)).
\]
The last estimate depends only on the constant $\alpha$ and the final time $T$.
\end{proof}

\begin{proof}[Proof of the Proposition~\ref{prp.bound.f}]
Due to Remark~\ref{ass.lip}, Proposition~\ref{prp.bound.traj} and Proposition~\ref{prp.bound.w2sol}, we have that
\[
\begin{split}
\|f(\traj{s}{x}{r},m_r) - f(x,m_s)\|
&\leq C_0 (\|\traj{s}{x}{r} - x\| + W_p(m_s,m_r)) \\
&\leq C_0(G_3(T) \cdot (1 + \|x\| + G_1(T,\alpha)) + G_4(T,\alpha)) \cdot (r-s).
\end{split}
\]
Evaluating the $L_p(\rd,m_s;\rd)$-norm of a both sides and applying Minkowski inequality with Proposition~\ref{prp.bound.sol} to the right-hand side, we arrive to the inequality
\[
\begin{split}
&\left( \int_\rd \|f(\traj{s}{x}{r},m_r) - f(x,m_s)\|^p\ m_s(dx) \right)^{1/p}\\
&\hspace{2cm}\leq C_0 \cdot \left( G_3(T) \cdot \left(1 + \left( \int_\rd \|x\|^p\ m_s(dx) \right)^{1/p} + G_1(T,\alpha) \right) + G_4(T,\alpha) \right) \cdot (r-s)\\
&\hspace{2cm}\leq C_0 \cdot \left( G_3(T) \cdot \left(1 + 2G_1(T,\alpha) \right) + G_4(T,\alpha) \right) \cdot (r-s).
\end{split}
\]
The last estimate depends only on the constant $\alpha$ and the final time $T$.
\end{proof}
%======================================================================================
\section{Properties of derivatives and superdifferentials in the space of probability measures} \label{app.diff}
%======================================================================================
\subsection{Properties of the inner derivative}
%======================================================================================
\begin{proof}[Proof of the Proposition~\ref{prp.flat.intr}]
First, notice that the left-hand side of the~\eqref{eq.flat.intr} can be rewritten as
\[
\int_\rd \frac{\delta \Phi}{\delta m}((1-s)m_*+sm,y)\ m(dy)
- \int_\rd \frac{\delta \Phi}{\delta m}((1-s)m_*+sm,y_*)\ m_*(dy_*).
\]
Here we denoted the integration variable in the second term by $y_*$.
The first term can be transformed in the following form:
\[
\begin{split}
\int_\rd \frac{\delta \Phi}{\delta m}(&(1-s)m_*+sm,y)\ m(dy)\\
&= \int_\rd \int_\rd \frac{\delta \Phi}{\delta m}((1-s)m_*+sm,y)\ \pi(dy_*|y)\ m(dy)\\
&= \int_\rd \frac{\delta \Phi}{\delta m}((1-s)m_*+sm,y)\ \pi(dy_*dy).
\end{split}
\]
Similarly, the second term takes the form
\[
\begin{split}
\int_\rd \frac{\delta \Phi}{\delta m}(&(1-s)m_*+sm,y_*)\ m_*(dy_*)\\
&= \int_\rd \int_\rd \frac{\delta \Phi}{\delta m}((1-s)m_*+sm,y_*)\ \pi(dy|y_*)\ m_*(dy_*)\\
&= \int_\rd \frac{\delta \Phi}{\delta m}((1-s)m_*+sm,y_*)\ \pi(dy_*dy).
\end{split}
\]
Thus, left-hand side of the~\eqref{eq.flat.intr} can be rewritten as
\[
\begin{split}
\int\limits_{\rd\times\rd} \left( \frac{\delta \Phi}{\delta m}((1-s)m_*+sm,y)
- \frac{\delta \Phi}{\delta m}((1-s)m_*+sm,y_*) \right)\ \pi(dy_*dy).
\end{split}
\]
Due to the Lagrange theorem, we have that
\[
\begin{split}
\frac{\delta \Phi}{\delta m}((1-s)m_*+sm,y)
&- \frac{\delta \Phi}{\delta m}((1-s)m_*+sm,y_*)\\
&= \int_0^1 \nabla_y \frac{\delta \Phi}{\delta m}((1-s)m_*+sm,y_*+q(y-y_*))\ dq \cdot (y-y_*) \\
&= \int_0^1 \dm \Phi((1-s)m_*+sm,y_*+q(y-y_*))\ dq \cdot (y-y_*).
\end{split}
\]
Substituting this expression to the previous formula we obtain Proposition~\ref{prp.flat.intr}.
\end{proof}

\begin{proof}[Proof of the Proposition~\ref{prp.intr.to.bar}]
Due to~\cite[Proposition~2.2.3]{Cardaliaguet_2019}, we have that
\[
\lim\limits_{\tau \to 0} \frac{\phi((\Id+\tau b)\sharp m) - \phi(m)}\tau
= \int_\rd \dm \phi(m,y) \cdot b(y)\ m(dy).
\]
This is equivalent to
\[
\frac{\phi((\Id+\tau b)\sharp m) - \phi(m)}\tau
= \int_\rd \dm \phi(m,y) \cdot b(y)\ m(dy) + \xi(\tau),
\]
where $\xi:\rr \to \rr$ is such that $\xi(\tau) \to 0$ as $\tau \to 0$.
Then
\[
\phi((\Id+\tau b)\sharp m) - \phi(m)
= \int_\rd \dm \phi(m,y) \cdot \tau b(y)\ m(dy) + \tau \xi(\tau).
\]
\end{proof}
%======================================================================================
\subsection{Strong Frechet and barycentric superdifferentials}
%======================================================================================
In this section, we assume that
\begin{itemize}
\item $X$ is Banach space;
\item $q = p' = \frac{p}{p-1}$.
\end{itemize}

\begin{dfn}
\[
\pp_{pq}(X \times X) \triangleq \{ \nu \in \pp(X \times X):\ \proj^1\sharp\nu \in \pp_p(X), \proj^2\sharp\nu \in \pp_q(X) \}
\]
\end{dfn}

\begin{dfn}
We say that $\mbmu \in \pp(X \times X \times X)$ is \dfnem{3-plan} for measures $\nu \in \pp_{pq}(X \times X)$ and $\mu^3 \in \pp_p(X)$ if
\[
\begin{split}
\proj^{1}\sharp\nu &= \mu_1,\\
\proj^{1,2}\sharp\mbmu &= \nu,\\
\proj^{1,3}\sharp\mbmu &\in \Pi(\mu_1,\mu_3),
\end{split}
\]
where $\mu_1 \in \pp_p(X)$.

The set of all 3-plans will be denoted by $\Pi^3(\nu,\mu^3)$.
\end{dfn}

\begin{dfn}[{\cite[Definition~10.3.1]{Ambrosio_2005}}] \label{dfn.supd.ambrosio}
Let $\mu_1 \in \pp_p(X)$ and a functional $\phi : \pp_p(X) \to \rr$ is upper semi-continuous.
Then, a measure $\alpha \in \pp_{pq}(X \times X)$ is an element of the \dfnem{strong Frechet superdifferential} $\partial^+ \phi(\mu^1)$ of the function $\phi$ at the point $\mu^1$ if for every measure $\mu^3 \pp_p(X)$ and any 3-plan $\mbmu \in \Pi^3(\alpha,\mu^3)$, there exists a function $\zeta: \rr_+ \to \rr$ such that
\begin{itemize}
\item $\zeta(\tau) \to 0$ as $\tau \to 0$;
\item the following condition holds:
\[
\phi(\mu^3) - \phi(\mu^1)
\leq \int_{X^3} x_2^\tran \cdot (x_3-x_1)\ \mbmu(dx_1dx_2dx_3) + W_p(\mu^1,\mu^3) \zeta(W_p(\mu^1,\mu^3)).
\]
\end{itemize}
\end{dfn}

\begin{prp} \label{prp.ambr.to.bar}
Let $m \in \pp_2(X)$ and let a function $\phi: \pp_2(X) \to \rr$ have non-empty strong Frechet superdifferential $\partial^+\phi(m)$.
Then, for every $\alpha \in \partial^+\phi(m)$, its barycenter $\int_X x_2\ \alpha(dx_2|x_1)$ is an element of the barycentric superdifferential $\supd\phi(m)$.
\end{prp}
\begin{proof}
We choose $\alpha \in \partial^+\phi(m)$.
Given $i>0$ and $b \in L_2(X,\mu;X)$ we denote
\[
\begin{split}
\mu_1 &\triangleq m,\\
\mu_3 &\triangleq (\Id+\tau b)\sharp m,\\
\mbmu &\triangleq (\Id,\ \proj^1 + \tau b \circ \proj^1) \sharp \alpha.
\end{split}
\]
Then,
\[
W_p(\mu_1,\mu_3) \leq \tau \|b\|_{L_2(X,\mu_1;X)}.
\]
Further, we denote
\[
\xi(\tau) \triangleq \|b\|_{L_2(X,\mu_1;X)} \cdot \zeta(W_p(\mu_1,\mu_3)).
\]
Due to the construction, we have that $\xi(\tau) \to 0$ as $\tau \to 0$.
Therefore, by Definition~\ref{dfn.supd.ambrosio},
\[
\begin{split}
\phi(\mu_3) - \phi(\mu_1)
&\leq \int_{X^3} x_2^\tran \cdot (x_3-x_1)\ \mbmu(dx_1dx_2dx_3) + W_p(\mu_1,\mu_3) \zeta(W_p(\mu_1,\mu_3))\\
&= \int_{X^2} x_2^\tran \cdot \tau b(x_1)\ \alpha(dx_1dx_2) + \tau \xi(\tau).
\end{split}
\]
Finally, disintegrating the measure $\alpha$, we obtain
\[
\phi(\mu_3) - \phi(\mu_1)
\leq \int\limits_X \left( \int_X x_2^\tran\ \alpha(dx_2|x_1), \tau b(x_1) \right)\ \mu_1(dx_1) + \tau \xi(\tau).
\]
\end{proof}
%======================================================================================
\section{Boundedness of the barycentric differential} \label{app.bound}
%======================================================================================
In this section we analyze the uniform boundedness of barycentric superdifferentials in the case of $p = 2$.

\begin{prp}
Let $m \in \pp_2(\rd)$ and let $\alpha>0$ be such that $\varsigma_2(m) \leq \alpha$.
Assume, additionally, that $\phi : \pp_2(\rd) \to \rr$ is locally Lipschitz continuous and superdifferentiable at the point $m$.
Then, for all $\gamma \in \supd\phi(m)$, the following estimation holds:
\[
\|\gamma\|_{L_2(\rd,m;\rd)} \leq K_\alpha,
\]
where $K_\alpha$ is Lipschitz constant of the $\phi$ in the ball of the radius $\alpha$.
\end{prp}
\begin{proof}
We let in the definition of the superdifferential $b = -\gamma^\tran$.
Then, there exists a function $\xi: \rr \to \rr$ such that, for all $\tau > 0$,
\[
\phi((\Id - \tau \gamma^\tran)\sharp m) - \phi(m)
\leq - \tau \int_{\rd} \gamma^\tran(x) \cdot \gamma(x)\ m(dx) + \tau \xi(\tau)
= - \tau \| \gamma \|_{L_2(\rd,m;\rds)}^2 + \tau \xi(\tau).
\]
Using Proposition~\ref{prp.bound.mom}, we have that
\[
\forall \epsilon > 0\ \exists \theta > 0:\ \forall \tau \in [0,\theta]\ \ \varsigma_2((\Id - \tau \gamma)\sharp m) \leq \alpha+\epsilon.
\]
This and the local Lipschitz continuity imply the following estimate:
\[
\begin{split}
\tau \|\gamma\|_{L_2(\rd,m;\rds)}^2 - \tau\xi(\tau)
&\leq -\big( \phi((\Id - \tau \gamma^\tran)\sharp m) - \phi(m) \big)
\leq \big| \phi((\Id - \tau \gamma^\tran)\sharp m) - \phi(m) \big| \\
&\leq K_{\alpha+\epsilon} \cdot W_2((\Id - \tau \gamma^\tran)\sharp m,m)
\leq K_{\alpha+\epsilon} \cdot \tau \left( \int_{\rd} \|\gamma(x)\|^2 m(dx) \right)^{1/2}.
\end{split}
\]
Further, dividing by $\tau$ and letting $\tau \to 0$, we obtain
\[
\|\gamma\|_{L_2(\rd,m;\rds}^2 \leq K_{\alpha+\epsilon} \|\gamma\|_{L_2(\rd,m;\rds)}.
\]
Finally, letting $\epsilon \to 0$ and using the right continuous of the Lipschitz constant, we conclude that
\[
\|\gamma\|_{L_2(\rd,m;\rds)}^2 \leq K_\alpha \|\gamma\|_{L_2(\rd,m;\rds)}.
\]
\end{proof}

Similarly, an analogous statement can be proved for subdifferentials.

\begin{prp}
Let $m \in \pp_2(\rd)$ and let $\alpha>0$ be such that $\varsigma_2(m) \leq \alpha$.
Assume, additionally, that $\phi : \pp_2(\rd) \to \rr$ is locally Lipschitz continuous and subdifferentiable at the point $m$.
Then for all $\gamma \in \subd\phi(m)$ the following estimation holds:
\[
\|\gamma\|_{L_2(\rd,m;\rd)} \leq K_\alpha,
\]
where $K_\alpha$ is Lipschitz constant of the $\phi$ in the ball of the radius $\alpha$.
\end{prp}
%======================================================================================
\printbibliography
%======================================================================================
\end{document}